\documentclass[12pt]{amsart}
\usepackage[utf8]{inputenc}
\usepackage{fullpage}
\usepackage{color}
\usepackage{graphicx}
\usepackage{amsmath,amsthm,amssymb,amsfonts}
\usepackage{enumerate}

\begin{document}

\def\baire{{\nat}^{\nat}}
\def\ideal{{\mathcal I}}
\def\binary{2^{< \omega}}
\def\nat{\mathbb{N}}
\def\cantor{2^{\nat}}
\def\fin{\mbox{\sf Fin}}
\def\filter{{\mathcal F}}
\def\ppoint{$\mbox{p}^+$}
\def\Pm{\mbox{$ \text{p}^-$}}
\def\qpoint{$\mbox{q}^+$}
\def\nin{\not\in}
\def\nwd{\mbox{\sf nwd}}
\def\fsig{F_\sigma}
\def\X{\mathbb{X}}
\def\Y{\mathbb{Y}}

\newtheorem{definition}{Definition}[section]
\newtheorem{theorem}[definition]{Theorem}
\newtheorem{example}[definition]{Example}
\newtheorem{corollary}[definition]{Corollary}
\newtheorem{lemma}[definition]{Lemma}
\newtheorem{proposition}[definition]{Proposition}
\newtheorem{question}[definition]{Question}
\newtheorem{claim}[definition]{Claim}

\title{Combinatorial properties on nodec countable spaces with analytic topology}

\author{Javier Murgas and Carlos Uzc\'ategui}
\address{Escuela de Matem\'aticas, Facultad de Ciencias, Universidad Industrial de
	Santander, Ciudad Universitaria, Carrera 27 Calle 9, Bucaramanga,
	Santander, A.A. 678, COLOMBIA. }
\email{javier\_murgas@hotmail.com.
}
\address{Escuela de Matem\'aticas, Facultad de Ciencias, Universidad Industrial de
	Santander, Ciudad Universitaria, Carrera 27 Calle 9, Bucaramanga,
	Santander, A.A. 678, COLOMBIA. Centro Interdisciplinario de L\'ogica y \'Algebra, Facultad de Ciencias, Universidad de Los Andes, M\'erida, VENEZUELA.}
\email{cuzcatea@saber.uis.edu.co.}
\thanks{The second author thanks Vicerrector\'ia de Investigaci\'on y Extensi\'on de la Universidad Industrial de Santander for the financial support for this work,  which is part  of the VIE project  \#2422.}

\date{}

\begin{abstract}We study some variations of the  product topology on families of clopen subsets of $\cantor\times\nat$ in order to construct countable nodec regular spaces (i.e. in which every nowhere dense set is closed) with  analytic topology which in addition are not selectively separable and do not satisfy the  combinatorial principle  $q^+$. 
\end{abstract}

\maketitle



\section{Introduction}

A topological space $X$ is {\em selectively separable} ($SS$), if for
any sequence $(D_n)_n$ of dense subsets of $X$ there is a finite set 
$F_n\subseteq D_n$,  for $n\in\nat$,  such that $\bigcup_n F_n$ is dense in
$X$. This notion was introduced by Scheepers \cite{Scheeper99}  and  has received a lot of attention ever since (see for instance \cite{BarmanDow2011,BarmanDow2012,Bella2009,Bella_et_al2008,Bella2013,CamargoUzca2018b,Gruenhage2011,Reposvetal2010}). Bella et al. \cite{Bella_et_al2008}  showed that every separable space with countable fan tightness is $SS$. On the other hand, Barman and Dow \cite{BarmanDow2011}  showed that every separable Fr\'echet space is also $SS$ (see also \cite{CamargoUzca2018b}). 

A topological space is  {\em maximal}  if it is a dense-in-itself regular space such that any strictly finer topology has an isolated point.
It was shown by van Douwen \cite{Vand} that a space is maximal if, and only if, it is {\em extremely disconnected} (i.e. the closure of every open set is open), {\em nodec} (i.e. every nowhere dense set is closed)  and every open set is {\em irresolvable} (i.e. if $U$ is open and $D\subseteq U$ is dense in $U$, then $U\setminus D$ not dense in $U$). He constructed a countable maximal regular space. 

A countable space $X$  is \qpoint\ at a point $x\in X$, if given any collection of finite sets $F_n\subseteq X$ such that $x\in \overline{\bigcup_n F_n}$, there is $S\subseteq \bigcup_n F_n$ such that  $x\in \overline{S}$ and  $ S\cap F_n$ has at most one point  for each $n$. We say that $X$ is a {\em \qpoint-space} if it is \qpoint\ at every point.  Every countable sequential space is \qpoint\ (see \cite[Proposition 3.3]{Todoruzca2000}). 
The collection of clopen subsets of $\cantor$ with the product topology is not \qpoint\ at any point. This notion is motivated by the analogous concept of a \qpoint\ filter (or ideal) from Ramsey theory. 

A problem stated in \cite{Bella_et_al2008} was to analyze the behavior of selective separability on maximal spaces. The existence of a maximal regular SS space  is independent of ZFC. In fact, in ZFC  there is a maximal non SS space \cite{BarmanDow2011} and  it is consistent with ZFC that no countable maximal space is SS \cite{BarmanDow2011, Reposvetal2010}. On the other hand, it is also consistent that there is a maximal,  countable, SS regular space  \cite{BarmanDow2011}. 

In this paper  we are interested in these properties on  countable spaces with an analytic topology (i.e. the topology of the space $X$ is an analytic set  as a subset of $2^X$ \cite{todoruzca}).   Maximal topologies are not analytic.  In fact, in \cite{Todoruzca2014} it was shown that there are neither extremely disconnected nor irresolvable analytic topologies, nevertheless there are nodec regular spaces with analytic topology. In view of the above mentioned results about maximal spaces, it seems natural to wonder about the behavior of selective separability on nodec spaces with an analytic topology. Nodec regular spaces are not easy to construct. We continue the study of the method introduced in \cite{Todoruzca2014} in order to construct similar  nodec regular  spaces with analytic topology that are neither SS nor \qpoint.  A countable regular space has an analytic topology if, and only if,  it is homeomorphic to a subspace of $C_p(\baire)$ \cite{todoruzca}. Thus our examples are constructed using some special topologies on a collection of clopen subsets of $\cantor\times \nat$. It is an open question whether there is a nodec $SS$ regular space with analytic topology. 

\section{Preliminaries}

An {\em ideal} on a set $X$ is a collection $\ideal$ of subsets of
$X$ satisfying: (i) $A\subseteq B$ and $B\in \ideal$, then $A\in \ideal$. (ii) If $A,B\in\ideal$, then  $A\cup B\in \ideal$. (iii)  $\emptyset \in \ideal$.  We will always assume that an ideal contains all finite subsets of $X$.  If $\ideal$ is an ideal on $X$, then $\ideal^+=\{A\subseteq X:\, A\nin \ideal\}$. 
 \fin\ denotes the ideal of finite subsets of the non negative integers $\nat$. An ideal $\ideal$ on $X$ is {\em tall}, if for every $A\subseteq X$ infinite, there is $B\subseteq A$ infinite with $B\in \ideal$.  We denote by $A^{<\omega}$ the collection of finite sequences of elements of $A$. If $s$ is a finite sequence on $A$ and $i\in A$, $|s|$ denotes its length and $ s\widehat{\;\;}i$ the sequence obtained concatenating $s$ with $i$.   For $s\in\binary$ and $\alpha\in \cantor$, let $s\prec \alpha$ if $\alpha(i)=s(i)$ for all $i<|s|$ and  
 $$
 [s]=\{\alpha\in \cantor: \; s\prec \alpha\}.
 $$
 If $\alpha\in\cantor$ and $n\in \nat$, we denote by $\alpha\restriction n$ the finite sequence $(\alpha(0),\cdots,\alpha(n-1))$ if $n>0$ and  $\alpha\restriction 0$ is the empty sequence. The collection of all $[s]$  with $s\in\binary$ is a basis of clopen sets for $\cantor$.  As usual we identify each $n\in \nat$ with $\{0,\cdots, n-1\}$. 

The ideal of nowhere dense subsets of $X$ is denoted by $\nwd(X)$.  
Now we recall some combinatorial properties of ideals.  We put $A\subseteq^*B$ if $A\setminus B$ is finite.
\begin{enumerate}

\item[({$p^+$})] $\ideal$ is \ppoint, if for every decreasing sequence $(A_n)_n$ of sets in $\ideal^+$, there is $A\in \ideal^+$ such that $A\subseteq^* A_n$ for all
$n\in\nat$. Following  \cite{HMTU2017}, we say that $\ideal$   is $\Pm$, if for every decreasing sequence $(A_n)_n$ of sets in $\ideal^+$ such that $A_n\setminus A_{n+1}\in \ideal$, there is $B\in \ideal^+$ such that $B\subseteq^* A_n$ for all $n$.

\item[($q^+)$] $\ideal$ is \qpoint\ , if for every $A\in \ideal^+$ and every
partition $(F_n)_n$ of $A$ into finite sets, there is $S\in\ideal^+$
such that $S\subseteq A$ and $S\cap F_n$  has at most one element for
each $n$. Such sets $S$ are called (partial) {\em selectors} for the partition. 
\end{enumerate}

A point $x$ of a topological space $X$ is called a {\em Fr\'echet
point}, if for every $A$ with $x\in \overline{A}$  there is a sequence $(x_n)_n$ in $A$
converging to $x$. We will say
that $x$ is a \qpoint-{\em point}, if $\ideal_x$ is  \qpoint.   We say that a space is a \qpoint-space, if every point is \qpoint. We define analogously the   notion of a \ppoint and $\Pm$ points. Notice that if $x$ is isolated, then $\ideal_x$ is  trivially \qpoint\ as $\ideal_x^+$ is empty. Thus  a space is \qpoint\  if, and only if, $\ideal_x$ is \qpoint\ for every non isolated point $x$. The same occurs with the other combinatorial properties defined in terms of $\ideal_x$. 

We say that a space $Z$ is {\em wSS} if for every sequence $(D_n)_n$ of dense subsets of $Z$, there is $F_n\subseteq D_n$ a finite set, for each $n$, such that $\bigcup_n F_n$ is not nowhere dense in $Z$.  In the terminology of selection principles \cite{Scheeper99}, $wSS$ corresponds to $S_{fin}(\mathcal{D}, \mathcal{B})$ where $\mathcal{D}$ is the collection of dense subsets and $\mathcal{B}$  the collection of non nowhere dense sets. Seemingly this notion has not been considered before. Notice that if $Z$ is $SS$ and $W$ is not $SS$, then the direct sum of $Z$ and $W$ is $wSS$ but not $SS$. 

A subset $A$ of a Polish space is called {\em analytic}, if it is a
continuous image of a Polish space. Equivalently, if there is a
continuous function $f:\baire\rightarrow X$ with range $A$, where
$\baire$ is the space of irrationals.    For
instance, every Borel subset of a Polish space is analytic. A general reference for all descriptive set theoretic notions used in this paper is \cite{Kechris94}. We say
that a topology $\tau$ over a countable set $X$ is {\em analytic},
if $\tau$ is analytic as a subset of the cantor cube $2^X$
(identifying subsets of $X$ with characteristic functions)
\cite{todoruzca, Todoruzca2000,Todoruzca2014}, in this case we will say that $X$ is an {\em analytic space}. A regular countable space is analytic if, and only if, it is homeomorphic to a subspace of $C_p(\baire)$ (see \cite{todoruzca}). If there is a base $\mathcal B$ of $X$ such that $\mathcal B$ is an $F_\sigma$ (Borel) subset of
$2^X$, then we say that $X$ has an {\em $F_\sigma$ (Borel) base}. In general, if $X$ has a Borel base, then the topology of $X$ is analytic.

We end this section recalling some results about countable spaces that will be used in the sequel. 

\begin{theorem}
\label{Fsesp} \cite[Corollary 3.8]{CamargoUzca2018b}
Let $X$ be a countable space with an $F_{\sigma}$ base, then  $X$ is $p^+$.
\end{theorem}

Next result is essentially  Lemma 4.6 of \cite{Todoruzca2014}. 

\begin{lemma}\label{scompact}
Let $X$ be a  $\sigma$-compact space and $W$ a countable collection of clopen subsets of  $X$. Then $W$, as a subspace of  $2^{X}$,  has an  $F_{\sigma}$ base.
\end{lemma}

\begin{theorem}\label{pesSS}\cite[Theorem 3.5]{CamargoUzca2018b}
Let $X$ be a countable space. If $X$ is $\Pm$, then  $X$ is $SS$. In particular, if  $X$ has an $F_{\sigma}$  base, then  $X$ is  $SS$.
\end{theorem}

A space  $X$ is {\em discretely generated} (DG) if
for every $A\subseteq X$ and $x\in\overline A$, there is $E\subseteq A$ discrete such that $x\in \overline E$. This notion was introduced by Dow et al. in \cite{DTTW2002}. It is not easy to construct spaces which are not DG, the typical examples are maximal spaces (which are  nodec).

\begin{theorem}
\label{sq-disc-generated} Let $X$ be a  regular countable space. Suppose  every non isolated point is \Pm, then  $X$ is discretely generated.
\end{theorem}

\proof Let $A\subset X$ with $x\in \overline A$. Fix a maximal family $(O_n)_n$ of
relatively open disjoint subsets of $A$ such that $x\nin
\overline{O_n}$. Let $B=\bigcup_n O_n$. From the maximality we get that
$x\in \overline{B}$.  Since each $O_n$ does not accumulate to $x$ and $x$
is a \Pm-point, there is $E$ such that $x\in \overline{E}$ and $E\cap O_n$ is
finite for every $n$. Clearly $E$ is a discrete subset of $A$.
\qed

\begin{theorem}(Dow et al \cite[Theorem 3.9]{DTTW2002})
\label{seq-disc-generated} Every Hausdorff sequential space is discretely generated. 
\end{theorem}

In summary, we have the  following implications  for  countable regular spaces (see  \cite{CamargoUzca2018b}).

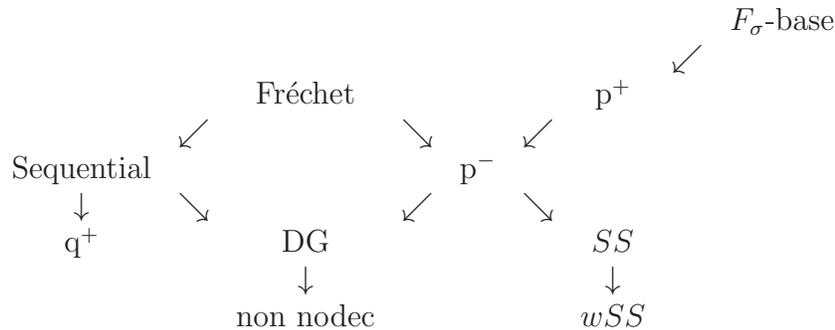
\[
\begin{array}{cccccccccclcl}
&&&&&&& \\
& & &  &  && &  &\fsig\mbox{-base}\\
 &&&                                   &&&&\swarrow\\
 &  &                     \mbox{Fr\'echet} &&  & & \mbox{\ppoint} & \\
                  &\swarrow & &\searrow   && \swarrow&&   \\
\mbox{Sequential} &  && &\mbox{\Pm}\\
\downarrow        & \searrow &&\swarrow  && \searrow&\\
\mbox{\qpoint} &&  \mbox{DG}& & &&\mbox{$SS$} && \\
               &&  \downarrow& & && \downarrow && \\
               &&  \text{non nodec}& & &&\text{$wSS$} && 
\end{array}
\]
\bigskip

\subsection{A $SS$, \qpoint\ nodec analytic non regular topology}
As we said in the introduction, nodec regular spaces are not easy construct. However, non regular nodec spaces are fairly easy to define.  We recall a well known construction given in \cite{Njastad1965}.  Let $\tau$ be a topology and define
\[
\tau^\alpha=\{V\setminus N:\: V\in\tau\;\mbox{and  $N\in\nwd(\tau)$}\}.
\]
Then $\tau^\alpha$ is a topology finer than $\tau$ (see \cite{Njastad1965}).

\begin{lemma}\cite{Njastad1965}
\label{tau-alpha} Let $(X,\tau)$ be a space.
\begin{itemize}

\item[(i)] $V\in\tau^\alpha$ iff $V\subseteq int_\tau
(cl_\tau(int_\tau(V)))$.

\item[(ii)] Let $A\subseteq X$ and $x\nin A$. Then
$x\in cl_{\tau^\alpha} (A)$ if, and only if, $x\in cl_\tau(int_\tau(cl_\tau(A)))$. 

\item[(iii)] $(X,\tau^\alpha)$ is a nodec space.
\end{itemize}
\end{lemma}

\begin{proposition}
\label{tau-alpha-q-point}
Let $(X,\tau)$  be a countable space.
\begin{itemize}
\item[(i)] If $(X,\tau)$ is Fr\'echet, then $(X,\tau^\alpha)$ is
a \qpoint-space.

\item[(ii)] $(X,\tau)$ is $SS$ if, and only if, $(X,\tau^\alpha)$ is $SS$.
\end{itemize}
\end{proposition}

\proof (i) Suppose $x\in cl_\alpha(A)\setminus A$ and $(F_n)_n$ is
a partition of $A$ with each $F_n$ finite. Let
$V=int_\tau(cl_\tau(A))$. By Lemma \ref{tau-alpha} we have
that $x\in cl_\tau(V)$. Let $(y_m)_m$ be an enumeration of $V$.
Since $A$ is $\tau$-dense in $V$, for every $m$ there is a
sequence $(x^m_i)_i$ in $A$ such that $x^m_i\rightarrow y_m$ when $i
\to \infty$
(with respect to $\tau$). Since each $F_n$ is finite, we can assume (by
passing to a subsequence if necessary) that each $(x^m_i)_i$ is a
selector for the partition $(F_n)_n$. Let $S_m$ be the range of
$(x^m_i)_i$. Notice that $x\nin cl_\tau(S_m)$ and every infinite
subset of $S_m$ is also a selector for $(F_n)_n$.  By a
straightforward diagonalization, for each $m$,  there is  $T_m\subseteq
S_m$ such that each $T_m$ is a selector and moreover $\bigcup_m
T_m$ is also a selector. Hence we can assume that
$S=\bigcup_m\{x^m_i:\;i\in\nat\}$ is a selector for the partition.
But clearly  $S$ is $\tau$-dense in $V$ and thus $V\subseteq
int_\tau(cl_\tau(S))$. Hence $x\in cl_\alpha(S)$ (by Lemma
\ref{tau-alpha}(i)).

(ii) By Lemma \ref{tau-alpha}(ii), a set is $\tau$-dense iff it is
$\tau^\alpha$-dense. \qed

\bigskip

Let $\tau$ be the usual metric topology on the rational ${\mathbb Q}$. It is not difficult to verify that $\tau^\alpha$ is analytic (in fact, it is Borel) and non regular (see \cite{Todoruzca2014}). Thus $({\mathbb Q}, \tau^\alpha)$ is a SS, \qpoint and nodec non regular space with  analytic topology. It is not known if there is a regular space with the same properties.

\section{The spaces $\X(\ideal)$ and  $\Y(\mathcal{I})$}

We recall the definitions of the spaces $\X(\ideal)$ and  $\Y(\ideal)$ for  an ideal $\ideal$, which were introduced in  \cite{Todoruzca2014}. 

For each  non empty ${\mathcal A}\subseteq \cantor$, let $\rho_{\mathcal A}$ be the topology on $2^{\cantor\times \nat}$ generated by  the following sets:
\begin{equation*}
(\alpha,p)^{+} = \{ \theta \in 2^{\cantor\times \nat}: \theta(\alpha,p)=1\}, \hspace{1.2cm} (\alpha,p)^{-} = \{ \theta \in 2^{\cantor\times \nat}: \theta(\alpha,p)=0 \}, 
\end{equation*}
with $\alpha\in \mathcal{A}$. 
A  basic $\rho_{\mathcal A}$-open set is as follows:
$$
V=\bigcap_{i=1}^{m} (\alpha_i,p_i)^+ \cap \bigcap_{i=1}^{n} (\beta_{i},q_i)^- 
$$
for some $\alpha_1,\cdots,\alpha_m,\beta_1,\cdots, \beta_{n} \in {\mathcal A}$, $p_1,,...,p_m,q_1,...,q_n \in \mathbb{N}$.
We always assume that   $(\alpha_i, p_i)\neq (\beta_j, q_j)$ for all $i$ and $j$, which is equivalent to saying that any set  $V$ as above  is not empty. 

Let  $\X$  be the collection of all finite unions of clopen sets of the form  $[s] \times \{n\}$ with  $n \in \mathbb{N}$ and  $s \in 2^{< \omega}$. We also include $\emptyset$ as an element of $\X$. As usual, we regard  $\X$ as a subset of $2^{2^{\mathbb{N}} \times \mathbb{N} }$.  Let  $\{\varphi_n : n \in \mathbb{N} \}$ be an enumeration of  $\X$ and for convenience we assume that $\varphi_0$ is $\emptyset$. Each $\varphi_n$, regarded  as a function   from $2^{\mathbb{N}}\times \mathbb{N}$ to $\{0,1\}$, is continuous.  Notice that $\X$ is a group with the symmetric difference as operation. 

Let  $\psi _n: 2^{\mathbb{N}} \times \mathbb{N} \to \{0,1\} $ be defined by
\begin{equation*}
\psi _n (\alpha,m)=\left\{
\begin{array}{cl}
\varphi_n (\alpha,m), & \text{if } \alpha(n)=0.  \\
1, & \text{if } \alpha(n)=1. \\
\end{array}
\right.
\end{equation*} 
Then $\psi_n$ is a continuous function.  
Let 
\[
\Y=\{\psi_n:\;n\in \nat\}.
\]

Given  $\mathcal{I}\subseteq \cantor$, we define

\begin{eqnarray*}
\X(\mathcal{I}) & = & (\X,\rho_\ideal),\\
\Y(\mathcal{I}) & = & (\Y,\rho_\ideal). 
\end{eqnarray*}
Also notice that $\X(\ideal)$ is a topological group. 

To each  $F \subseteq \mathbb{N}$, we associate two sets $F'\subseteq \X$ and $\widehat{F}\subseteq \Y$:
$$
F':= \{ \varphi_n:\; n \in F\},
$$
$$
\widehat{F}:= \{ \psi_n:\; n \in F\}.
$$
The topological similarities between $F'$ and $\widehat{F}$  are crucial to establish some properties of $\Y(\ideal)$.  

As usual, we identify a subset $A\subseteq \nat$ with its characteristic function. So from now on, an ideal $\ideal$ over $\nat$ will be also viewed as a subset of $\cantor$. 
The properties of $\Y(\ideal)$ naturally depend on the ideal $\ideal$.

\begin{lemma}
\label{complexity}
If $\ideal$ is analytic, then $\X(\ideal)$ and $\Y(\ideal)$ have analytic topologies. 
\end{lemma}

\begin{proof}
It is easy to see that the standard subspace subbases for $\X(\ideal)$ and $\Y(\ideal)$ are also analytic when $\ideal$ is analytic. Thus the topology is analytic (see \cite[Proposition 3.2]{todoruzca}). 
\end{proof}

\begin{theorem}
\label{fsigY}
If  $\mathcal{I}$ is  an $F_{\sigma}$ ideal over $\nat$, then $\Y(\mathcal{I})$ has an $F_{\sigma}$ base and thus it is SS and DG.
\end{theorem}

\begin{proof}
It follows from Lemma \ref{scompact} and Theorems \ref{pesSS} and \ref{sq-disc-generated}.
\end{proof}

The reason to study the space $\Y( \ideal)$ is the following theorem. Let   
$$
\mathcal{I}_{nd}:=\{ F \subseteq \mathbb{N}  : \{  \varphi_n : n \in F \} \text{ is nowhere dense  in  } \X  \}.
$$

\begin{theorem}\cite{Todoruzca2014} \label{yind}
$\Y(\mathcal{I}_{nd})$ is a nodec regular space without isolated points and with an analytic topology.
\end{theorem}

$\Y(\mathcal{I}_{nd})$ was so far the only space we knew with the properties stated above. We will present a generalization of this theorem showing other ideals $\ideal$ such that $\Y(\ideal)$ has the same properties.

\subsection{The space $\X(\ideal)$} 

We present some properties of the space $\X(\ideal)$ that will be needed later.   We are interested in whether $\X(\ideal)$ is DG, SS or \qpoint. We start with a general result which is proven as Theorem \ref{fsigY}. 

\begin{theorem}
If  $\mathcal{I}$ is an $F_{\sigma}$ ideal over $\nat$, then $\X(\mathcal{I})$ has an $F_{\sigma}$ base, and thus it is SS and DG.
\end{theorem}

We will show that $\X(\ideal)$ is not \qpoint\ except in the extreme case when $\ideal$ is $\fin$.  The key lemma to show this is the following result. 

\begin{lemma} \label{xnoq}
There is a pairwise disjoint family  $\{ A_n : n \in \mathbb{N} \}$ of finite subsets of $\X$ such that $\bigcup_{k \in E} A_k$ is dense in $\X$ (with the product topology) for any  infinite  $E \subseteq \mathbb{N}$.  Moreover,  for each infinite set $E\subseteq \nat$,   each selector  $S$  for the family  $\{ A_n : n \in E\}$ and each  $\varphi \notin S \cup \{ \emptyset \}$, there is  $p\in\nat$ and  $\alpha \in 2^{\mathbb{N}}$ such that $\alpha^{-1}(1) \subseteq^* E$, $\varphi \in (\alpha,p)^+$  and  $(\alpha,p)^+ \cap S$ is finite. 
\end{lemma}

\begin{proof} We say that a $\varphi\in \X$ has the property $(*^m)$, for $m\in\nat$, if there are $k \in \mathbb{N}$ and finite sequences  $s_i$, for $i=1,...,k$,  of  length  $m+1$ such that  $ \varphi = \bigcup _{i=1}^{k} [s_i]\times\{m_i\}$, $m_i\leq m$ and  $s_i \restriction m \neq s_j \restriction m$,  whenever $m_i=m_j$ (i.e.  $[s_j]\cup [s_i]$ is not a basic clopen set). Let
$$
A_m = \{ \varphi \in \X: \varphi \text{ has the property }(*^m)  \}.
$$
Let $E\subseteq \nat$  be an infinite set. We will show that $A:=\bigcup_{k \in E} A_k$ is dense in $2^{\cantor\times\nat}$.  
Let $V$ be a basic open set of $2^{\cantor\times\nat}$, let us say 
$$
V=\bigcap_{i=1}^{m} (\alpha_i,p_i)^+ \cap \bigcap_{i=1}^{n} (\beta_{i},q_i)^- 
$$
for some $\alpha_1,\cdots,\alpha_m,\beta_1,\cdots, \beta_{n} \in \cantor $, $p_1,\cdots,p_m,q_1,\cdots,q_n \in \mathbb{N}$.
We need to show that $V\cap A$ is not empty.
Pick $l$ large enough such that $l+1\in E$, $l+1>\max\{p_i,q_j: i\leq m, j\leq n\}$,  $\alpha_i\restriction l\neq \alpha_j\restriction l$ for all $i$ and $j$ such that $\alpha_i\neq \alpha_j$, $\beta_i\restriction l\neq \beta_j\restriction l$ for all $i$ and $j$ such that $\beta_i\neq \beta_j$ and $\alpha_i\restriction l\neq \beta_j\restriction l$ for all $i$ and $j$ such that $\alpha_i\neq\beta_j$.  Let $\varphi = \bigcup_{i=1}^{m} [\alpha_i \restriction (l+2)]\times\{p_i\}$. Then $\varphi$ belongs to $A_{l+1}\cap V$.

To see the second claim, let $E\subseteq\nat$ be an infinite set and let $S=\{z_n: n \in E \}$ be a selector, that is,   $z_n \in A_n$ for all  $n \in E$. Fix $\varphi \notin S \cup \{\emptyset\}$, say $\varphi= \bigcup_{i=1}^{l} [t_i]\times\{p_i\}$ for some $t_i\in 2^{<\omega}$ and $p_i\in \nat$.   The required  $\alpha $ is recursively defined as follows:
$$
\alpha (n)=\left\{
		\begin{array}{cl}
		t_1(n), & \text{if } n<|t_1|,  \\
		1, &\mbox{if  $n \geq |t_1|, n \in E$ and  $[(\alpha\restriction n) \widehat{\;\;} 0]\times\{p_1\}   \subseteq z_n$}, \\
		0, & \text{otherwise.}
		\end{array}
		\right.
$$
From the definition of the sets  $A_m$, it is  easily shown that  $(\alpha,p_1) \notin  \bigcup\{ z_k:  k \geq |t_1| \text{ and }k \in E \}$. Clearly   $(\alpha,p_1)^+ \cap S \subseteq \{ z_k : k < |t_1| \text{ and } k \in E \}$ is finite and $\varphi\in (\alpha, p_1)^+$.   Finally, it is also clear from the definition of $\alpha$ that $\alpha^{-1}(1) \subseteq^* E$.

\end{proof}

\begin{theorem}
\label{eqclq}
Let $\mathcal{I}$ be an ideal on $\mathbb{N}$. Then $\X(\ideal)$ is \qpoint\ at some  (every) point  if, and only if, $\ideal=\fin$. 
\end{theorem}

\begin{proof} If $\ideal =\fin$, then $\X(\ideal)$ has a countable basis and thus it is \qpoint\ at every point.  Since $\X(\ideal)$ is homogeneous (as it is  a topological group), then if $\X(\ideal)$ is \qpoint\ at some point, then it is \qpoint\ at every point.  Suppose  now  that there is $E \in \mathcal{I} \setminus \fin$.  We will show that $\ideal$ is not \qpoint\ at some point. Let $\{ A_n : n \in \nat\}$  be the sequence, given by Lemma \ref{xnoq}, of pairwise disjoint finite subsets of  $\X$ such that  $A:=\bigcup_{k \in E} A_k$ is dense in $\X$. Since the topology of $\X$ is finer than the topology of $\X(\ideal)$, then  $A$  is dense in  $\X(\ideal)$.  Let $\varphi \nin A\cup\{\emptyset\}$. We will show that $\X(\ideal)$ fails the property \qpoint\ at $\varphi$. Let $S$ be a selector of $\{ A_n : n \in E\}$. Let $\alpha\in \cantor$ and $p\in \nat$ be as in the conclusion of  Lemma \ref{xnoq}, that is, $\alpha^{-1}(1) \subseteq^* E$, $\varphi \in (\alpha,p)^+$  and  $(\alpha,p)^+ \cap S$ is finite. Notice that $\alpha\in\ideal$ and hence $\varphi$ is not in the $\rho_\ideal$-closure of $S$.  Hence $\X(\ideal)$ is not \qpoint\ at $\varphi$.	

\end{proof}

Now we look at the $SS$ property. The following result provides a method to construct dense subsets of $\X(\ideal)$.

\begin{lemma}
\label{DA}
For each $A\subseteq  \nat$  infinite,  let $\mathbf{D}(A)$ be the following subset of $\X$:
\[
\left\lbrace \bigcup_{i=0}^{k} [s_i] \times \{ m_i \} \in \X:\; A \cap   s_i^{-1}(0) \neq \emptyset\; \mbox{ for all } i\in \{0,...,k\},  k\in \mathbb{N}, s_i \in 2^{< \omega}  \right\rbrace \cup \{\emptyset\}.
\]
Then $A\in  \ideal$ if, and only if, $\mathbf{D}(A)$ is not dense in $\X(\ideal)$ if,and only if, $\mathbf{D}(A)$ is  nowhere dense and closed in $\X(\ideal)$.
\end{lemma}

\begin{proof}

We  first show that $\mathbf{D}(A)$ is closed for every $A\in \ideal$.  We shall show that the complement of $\mathbf{D}(A)$ is open in $\X(\ideal)$. Let $\varphi\in \X\setminus \mathbf{D}(A)$.  Since $\varphi\neq\emptyset$, we have that $\varphi= \bigcup_{i=1}^{k} [s_i] \times \{m_i\}$ and we can assume that $A \cap s_1^{-1}(0) = \emptyset$. Let 
$B=A \cup s_1^{-1}(1)$. Notice that $B\in \ideal$. Let $\beta$ be  the characteristic function of $B$.  Clearly $\beta\in [s_1]$ and thus $\varphi\in  (\beta,m_1)^+$. On the other hand, suppose that  $\varphi'= \bigcup_{i=1}^{l} [t_i] \times \{p_i\}\in (\beta,m_1)^+$. Assume that $\beta\in [t_1]$ and $p_1=m_1$, then $t_1^{-1}(0)\subset \beta^{-1}(0)$ and hence $t_1^{-1}(0)\cap  A=\emptyset$. This shows that $\varphi'\nin \mathbf{D}(A)$ and thus  $(\beta,m_1)^+\cap \mathbf{D}(A)=\emptyset$.

Now we  show that if $A\in \ideal$, then $\mathbf{D}(A)$ is nowhere dense.  Since $\mathbf{D}(A)$ is closed, it suffices to show that it has empty interior. Let $V$ be a basic  $\rho_\ideal$-open set. Let us say 
\begin{equation}
\label{basicopen1}
V=\bigcap_{i=1}^{m} (\alpha_i,p_i)^+ \cap \bigcap_{i=1}^{n} (\beta_{i},q_i)^- 
\end{equation}
for some $\alpha_1,\cdots,\alpha_m,\beta_1,\cdots, \beta_{n} \in \mathcal{I} $, $p_1,,...,p_m,q_1,...,q_n \in \mathbb{N}$. Recall that $(\alpha_i,p_i)\neq (\beta_j,q_j)$ for all $i\neq j$. 
Since $\beta_i\in \ideal$, then $\beta_i^{-1}(0)\neq\emptyset$ for all $i$. Let  $l=\max \{ \min(\beta_i^{-1}(0) ) : 1 \leq i \leq n \}$ and $t$ be the constant sequence 1 of length  $l$. Since $\X$ is clearly $\rho_\ideal$-dense, let  $\varphi\in V\cap \X$. Then $\varphi\cup ([t]
\times \{0\})\in V\setminus  \mathbf{D}(A)$.

Finally, we show that if $A\not\in \ideal$, then $\mathbf{D}(A)$ is dense. Let $V$ be a basic $\rho_\ideal$-open set as given by \eqref{basicopen1}.    Pick $l$ large enough such that $\alpha_i\restriction l\neq \alpha_j\restriction l$ for $i\neq j$, $\beta_i\restriction l\neq \beta_j\restriction l$ for $i\neq j$ and $\alpha_i\restriction l\neq \beta_j\restriction l$ for all $i$ and $j$ such that $\alpha_i\neq\beta_j$.  Then pick $k\geq l$ such that $k\geq \min (\alpha_i^{-1}(0)\cap A)$  for all $i\leq m$ (notice that $\alpha_i^{-1}(0)\cap A\neq \emptyset$ as $A\nin \ideal$ and $\alpha_i\in \ideal$). Let $s_i=\alpha_i\restriction k$ for $i\leq m$ and $\varphi=\bigcup_{i=1}^m [s_i]\times\{p_i\}$. Then $\varphi\in V\cap \mathbf{D}(A)$.

\end{proof}

We remind the reader that $F'$ denotes the set $\{\varphi_n:\; n\in F\}$ for each  $F\subseteq \nat$. 

\begin{theorem}\label{xinoSS}
Let $\ideal$ be an ideal over $\nat$.  If  $\ideal $ is  not $p^+$, then  $\X(\mathcal{I})$ is not $wSS$.
\end{theorem}

\begin{proof} Suppose that $\mathcal I$ is not $p^+$ and  fix a sequence $(A_n)_{n \in \mathbb{N}}$  of subsets of  $\mathbb{N}$ such that $A_n \notin \mathcal{I}$,   $n \in \mathbb{N}$,  and  $\bigcup_{n \in \mathbb{N}} F_n \in \mathcal{I}$ for all $F_n \subseteq A_n$ finite. 
		
Let  $D_n=  \mathbf{D}(A_n)$ as in Lemma \ref{DA}.  We show that  the property $wSS$  fails at the sequence $(D_n)_n$. Let $K_n \subseteq D_n$ be  a finite set for each $n$, we need to  show that $\bigcup_n K_n$ is nowhere dense in $\X(\ideal)$. Let us enumerate each $K_n$ as follows: 
$$
K_n= \left\lbrace \bigcup_{i=0}^{k_{n,l}} [s_i^{n, l}]\times\{p^{n,l}_i\}: l < |K_n| \right\rbrace.
$$
Let  $q_n > \max \{ |s_i^{n, l}|: l < |K_n|, i \leq k_{n,l}  \}$.      By hypothesis, $B=\bigcup _{n \in \mathbb{N}} (A_n  \cap \{0, \cdots, q_n \}) \in \mathcal{I}$. Let $\beta$ be  the characteristic function of $B$. We claim that for all $m\in \nat$
$$
(\beta,m)^+ \cap (\bigcup_{n \in \mathbb{N}} K_n) = \emptyset.
$$
Otherwise,  there are  $n \in \mathbb{N}$, $l < |K_n|$ and  $i \leq k_{n,l}$  such that  $\beta \in [s_i^{n, l}]$, that is,  $s_i^{n, l} \preceq \beta$. But this contradicts the fact that  $(A_n \cap \{0, \cdots, q_n\}) \cap  \left( s_i^{n, l} \right)^{-1} (0) \neq \emptyset$ for all $i $ and  $l$ (recall that $D_n=\mathbf{D}(A_n)$).
Thus $(\bigcup_{n \in \mathbb{N}} K_n)\cap (\bigcup_m (\beta,m)^+)=\emptyset$. Since $\bigcup_m (\beta,m)^+$ is $\rho_\ideal$-open dense,  $\bigcup_n K_n$ is $\rho_\ideal$-nowhere dense.
\end{proof}

\begin{proposition}
\label{converseque}
Let $\ideal$ be an ideal over $\nat$. Any element of $\X(\ideal)$ is a limit of a non trivial sequence. 
\end{proposition}

\begin{proof}
Since $\X(\ideal)$ is a topological group, it suffices to show that there is a sequence converging to $\emptyset$ (i.e. to $\varphi_0$).

Let $(\alpha_n)_{n \in \nat}$ be a sequence in  $2^{\nat}$ such that $\alpha_k \restriction (k+1) \neq \alpha_l \restriction (k+1)$ for each $k<l$.
Let $(x_n)_{n}$ be defined by $x_n= [\alpha_n \restriction (n+1)]\times\lbrace 0\rbrace$. 
Let $V$ be a neighborhood of $\emptyset$, namely, $V= \bigcap_{i=1}^{m} (\beta_i,n_i)^-$ for some $\beta_i\in \ideal$ and $n_i \in \nat$. We have that $\alpha_n \restriction (n+1) \not\preceq \beta_i$ for almost every $n$, therefore $x_n \in V$ and $x_n \to \emptyset$. 

\end{proof}

\begin{question}
When is $\X(\ideal)$ discretely generated?
\end{question}

\subsection{The space $c(\ideal)$}

It is natural to wonder what can be said if instead of  $\X$ we use the more familiar space $CL(\cantor)$ of all clopen subsets of $\cantor$.

Exactly as before we can define a space $c(\ideal)$ as follows.  

\begin{definition} Let $\mathcal{I}$ be an ideal over $\nat$ and  $c(\mathcal{I})$ be $(CL(2^{\mathbb{N}}), \tau_{\mathcal{I}} )$, where  $\tau_{\ideal}$ is generated by the following subbasis:
$$ 
\alpha^+=\{x \in CL(2^{\mathbb{N}}): \alpha \in x \} \hspace{1cm} \text{and}\hspace{1cm} \alpha^-=\{x \in CL(2^{\mathbb{N}}): \alpha \notin x \}, 
$$
where $\alpha \in \mathcal{I}$. 
\end{definition}

In fact, it is easy to see that $c(\ideal)$ is homeomorphic to $\{\bigcup_{i=0}^{k} [s_i] \times \{ 0 \} \in \X:  k\in \mathbb{N}, s_i \in 2^{< \omega} \}$ and by a simple modification of the proofs above we have the following. 

\begin{theorem}
Let $\mathcal{I}$ be an ideal over $\mathbb{N}$. Then $c(\ideal)$ is \qpoint\ at some  (every) point  if, and only if, $\ideal=\fin$. 
\end{theorem}

\begin{theorem}
Suppose that $\mathcal{I}$ is an ideal over $\mathbb{N}$. If $\ideal $ is not  $p^+$, then  $c(\mathcal{I})$ is not $wSS$.
\end{theorem}

\subsection{The space $\Y(\ideal)$} 

In this section we work with the space $\Y(\ideal)$ in order to construct nodec spaces. To that end we introduce an operation $\star$ on ideals. 
We remind the reader that to each $F \subseteq \mathbb{N}$ we associate the sets $F'= \{ \varphi_n:\; n \in F\}$ and  $\widehat{F}= \{ \psi_n:\; n \in F\}$.

\begin{definition}
Let $\ideal$ be a nonempty subset of $\cantor$. We define:
$$
\ideal ^{\star} =\{F \subseteq  \nat:\; F' \text{ is nowhere dense in }\X(\ideal) \}. 
$$
\end{definition}

Notice that $\ideal^{\star}$ is a free ideal and $\ideal_{nd}=(\cantor)^{\star}$. We are going to present several results that are useful to compare  $\X(\ideal)$ and $\Y(\ideal)$. 

The following fact will be used several times in the sequel. 

\begin{lemma}
\label{simetricdif}
Let $\ideal$ be an ideal over $\nat$. Let $V$ be a basic $\rho_\ideal$-open set. Then  
\[
\{n\in\nat: \varphi_n\in V \} \triangle \{n\in\nat: \psi_n\in V\}\in \ideal.
\]
\end{lemma}

\begin{proof}
Let $V$ be  a non empty  basic open set, that is, 
\begin{equation}
\label{basicopen}
V=\bigcap_{i=1}^{m} (\alpha_i,p_i)^{+} \cap \bigcap_{j=1}^{l} (\beta_j,q_j)^{-}.
\end{equation}
From the very definition of $\psi_n$ and viewing it as a clopen set, we have that 
$$
\psi_n=\varphi_n\cup (\{\alpha\in \cantor:\;\alpha(n)=1\}\times \nat).
$$
From  this we have the following:
\[
\{n\in\nat: \varphi_n\in V \} \setminus \{n\in\nat: \psi_n\in V\}\subseteq \bigcup_{j=1}^l \beta_j^{-1}(1)
\]
and 
\[
\{n\in\nat: \psi_n\in V \} \setminus \{n\in\nat: \varphi_n\in V\}\subseteq \bigcup_{i=1}^m \alpha_i^{-1}(1).
\]
Thus when  each $\alpha_i$ and each $\beta_j$ belongs to $\ideal$, the unions on the right also belong to $\ideal$.

\end{proof}

In the following we compare $\X(\ideal)$ and $\Y(\ideal)$  in terms of their dense and nowhere dense subsets.   Some results need that the ideals $\ideal$ and $\ideal^\star$ are comparable, i.e. $\ideal\subseteq \ideal^\star$ or $\ideal^\star \subseteq \ideal$, it is unclear whether this is always the case. 

We are mostly interested in crowded spaces. The following fact gives a sufficient condition for   $\Y(\ideal)$  to be crowded. 

\begin{lemma}\label{crowded}
Let $\ideal$ be an ideal on $\nat$. Then 
\begin{enumerate}
\item $\X$ is dense in $(2^{\cantor\times\nat},\rho_\ideal)$.

\item   $int_{\X(\ideal)} ( F')= \emptyset$, for all $F\in \ideal$ if, and only if,   $\Y$ is dense in $(2^{\cantor\times\nat},\rho_\ideal)$.

\item If $\ideal\subseteq \ideal^\star$, then $\Y$ is dense in $(2^{\cantor\times\nat},\rho_{\ideal})$.

\item If $\ideal^\star\subseteq \ideal$, then $\Y$ is dense in $(2^{\cantor\times\nat},\rho_{\ideal^\star})$.

\end{enumerate}   
\end{lemma}

\begin{proof}
(1) is clear. The {\em only if} part of (2) was shown in  \cite[Lemma 4.2]{Todoruzca2014}, but we include a  proof for the sake of completeness.  Let $V$ be a nonempty basic $\rho_\ideal$-open set. We need to find $n$ such that $\psi_n\in V$. From Lemma \ref{simetricdif} we have that 
\[
E=\{n\in\nat: \varphi_n\in V \text{ and } \psi_n\nin V\}\in \ideal.
\]
Since $int_{\X(\ideal)} ( E')=\emptyset$, there is $n$ such that $\varphi_n\in V$ and $n\nin E$.  Therefore $\psi_n\in V$.

For the {\em if} part, suppose that  $\Y$ is dense in $(2^{\cantor\times\nat},\rho_\ideal)$ and, towards a contradiction, that there is a nonempty basic $\rho_\ideal$-open set $V$ such that 
$F=\{n\in \nat:\; \varphi_n\in V\}$ belongs to $\ideal$. From this and Lemma \ref{simetricdif} the following set belongs to $\ideal$:
\[
E=F\cup \{n\in\nat: \varphi_n\nin V \text{ and }    \psi_n\in V\}.
\]
Let $\beta$ be the characteristic function of $E$. Since $V$ is a basic open set of the form \eqref{basicopen}, there is $m$ such that $V\cap (\beta, m)^-\neq\emptyset$.  Since $\Y$ is $\rho_\ideal$-dense, there is $n$ such that $\psi_n\in V\cap (\beta, m)^-$. Hence $\psi_n\nin (\beta, m)^+$ and, by the definition of $\psi_n$, we have that $\beta(n)=0$. Therefore $n\nin E$ and  $\psi_n\in V$, then $\varphi_n\in V$. Thus  $n\in F$,  a contradiction. 

(3)  follows immediately from (2). To see (4), it suffices to show that  $int_{\X(\ideal^\star)} ( F')= \emptyset$, for all $F\in \ideal^\star$. Let $F\in \ideal^\star$. By definition,  $F'$ is nowhere dense in $\X(\ideal)$. In particular $int_{\X(\ideal^\star)}( F')= \emptyset$, as $\ideal^\star\subseteq \ideal$.
\end{proof}

Now we show that the operation $\star$ is monotone. 

\begin{lemma}\label{IssinIs}
Let $\ideal$ and $\mathcal{J}$ be  ideals over $\nat$ with $\mathcal{J}\subseteq \ideal$. Then 
\begin{enumerate}
\item For every basic $\rho_\ideal$-open set $V$ of $2^{\cantor\times \nat}$ there are sets $W$, $U$ such that $V=W\cap U$,  $W$ is a  $\rho_{\mathcal{J}}$-open set and $U$ is a basic $\rho_\ideal$-open set which is also $\rho_{\mathcal{J}}$-dense. 

\item If $A\subseteq 2^{\cantor\times \nat}$ is $\rho_\mathcal{J}$-nowhere dense, then $A$ is 
$\rho_\mathcal{I}$-nowhere dense.

\item   $\mathcal{J}^\star\subseteq  \ideal^{\star}$.  Moreover, if $\mathcal{J}\subsetneq \ideal$, then $\mathcal{J}^\star\subsetneq  \ideal^{\star}$.
\end{enumerate}

\end{lemma}

\begin{proof}
(1) Let $V$ be  a basic open set, that is, 
\begin{equation}
\nonumber
V=\bigcap_{i=1}^{m} (\alpha_i,p_i)^{+} \cap \bigcap_{j=1}^{l} (\beta_j,q_j)^{-}.
\end{equation}
Notice that if every  $\alpha$ and $\beta$ belongs to $\ideal\setminus\mathcal{J}$, then $V$ is  $\rho_{\mathcal{J}}$-dense. Thus given such basic open set $V$ where every $\alpha$ and $\beta$ belongs to $\ideal$,  we can separate them and  form $W$ and $U$ as desired: For $W$, we use the $\alpha$'s and $\beta$'s belonging to $\mathcal J$ (put $W=2^{\cantor\times \nat}$ in case there is none in $\mathcal{J})$ and for $U$, we use the $\alpha$'s and $\beta$'s belonging to $\ideal\setminus\mathcal J$.

(2) Let $A\subseteq 2^{\cantor\times \nat}$ be a $\rho_\mathcal{J}$-nowhere dense set. Let $V$ be a basic $\rho_\ideal$-open set of $2^{\cantor\times \nat}$. Then $V=W\cap U$ where  $W$ and $U$ are as given by part (1).  As $A$ is $\rho_\mathcal{J}$-nowhere dense, there is a non empty $\rho_\mathcal{J}$-open set $W'\subseteq W$  such that $W'\cap A=\emptyset$. Since $W'$ is also $\rho_\ideal$-open and $U$ is $\rho_\mathcal{J}$-dense, then $U\cap W'$ is a non empty $\rho_\ideal$-open set disjoint from $A$ and contained in $V$. 

(3) Since $\X$ is dense in $2^{\cantor\times\nat}$, then $A\in \nwd(\X(\ideal))$ if, and only if, $A$ is nowhere dense in  $(2^{\cantor\times\nat},\rho_\ideal)$. From this and  (2) we immediately get that  $\mathcal{J}^\star\subseteq  \ideal^{\star}$.  Finally, notice that from Lemma \ref{DA}, we have that for $A\in \ideal\setminus \mathcal{J}$, the set $\mathbf{D}(A)$ is  nowhere dense in $\X(\ideal)$ and dense in $\X(\mathcal{J})$.
\end{proof}

Next result gives a sufficient condition for $\Y(\ideal^\star)$ to be  nodec. It is a generalization of a result from  \cite{Todoruzca2014}. 

\begin{lemma}  
\label{Lemanodec}
Let  $\ideal$ be an ideal over $\nat$ and $F\subseteq \nat$. 

\begin{enumerate}
\item   If  $F \in \mathcal{I}$, then  $\widehat{F}$ is closed discrete in $\Y(\mathcal{I})$.

\item Let $\ideal$ be such that $\ideal^{\star} \subseteq \ideal$. If $\widehat{F}$ is nowhere dense in $\Y(\ideal^{\star})$, then $F \in \ideal^{\star}$.
\item If $\ideal^{\star} \subseteq \ideal$, then $\Y(\ideal^{\star})$ is nodec.
\end{enumerate}
\end{lemma}

\begin{proof} 
(1) is Lemma 4.1 from  \cite{Todoruzca2014} we include the proof for the reader's convenience.

Since $\ideal$ is hereditary, it suffices to show that $\widehat{F}$ is closed for every $F\in \ideal$. Let $F\in\ideal$
and let $F$ denote also its  characteristic function.  Notice that for each $m\in \nat$,  if $C=\{n\in\nat:\; \psi_n\in (F,m)^+\}$, then
$\widehat{C}$ is closed in $\Y(\ideal)$. We claim that
\[
F=\bigcap_{m\in\nat}\{n\in\nat:\; \psi_n\in (F,m)^+\}.
\]
From this it follows that $\widehat{F}$ is closed in $\Y(\ideal)$.
To show the equality above,  let $n\in F$, then by the definition of $\psi_n$, we have that 
$\psi_n\in (F,m)^+$ for all $m\in\nat$. Conversely, suppose $n\nin F$ and let 
 $\varphi_n$ be $[s_1]\times \{m_1\}\cup\cdots \cup [s_k]\times
\{m_k\}$. Pick $m\not\in\{m_1,\cdots,m_k\}$, then
$\varphi_n\nin (F,m)^+$ and thus $\psi_n\nin (F,m)^+$ by the definition of $\psi_n$.

(2) is a generalization of  Lemma 4.3 of \cite{Todoruzca2014}.  Let  $\widehat{F}$ be  nowhere dense in $\Y(\ideal^\star)$ and suppose, towards a contradiction, that $F\nin \ideal^\star$. Let $V$ be a basic $\rho_{\ideal}$-open set such that $F'\cap V$ is $\rho_\ideal$-dense in $V$. 
By Lemma \ref{IssinIs}, there are sets $W$ and $U$ such that $V=W\cap U$, $W$ is a  $\rho_{\ideal^\star}$-open set, $U$ is a basic $\rho_\ideal$-open set and $U$ is also $\rho_{\ideal^\star}$-dense.  Since $\widehat{F}$ is nowhere dense in $\Y(\ideal^\star)$, there is a basic $\rho_{\ideal^\star}$-open set $W'\subseteq W$ such that $\widehat{F}\cap W'=\emptyset$, that is 
\[
F\cap\{n\in \nat:\; \psi_n\in W'\}=\emptyset.
\]
From Lemma \ref{simetricdif} we know that 
\[
\{n\in\nat: \varphi_n\in W' \} \setminus \{n\in\nat: \psi_n\in W'\}\in \ideal^\star.
\]
From this and the previous fact we get
\[
F\cap \{n\in \nat:\; \varphi_n\in W'\}\in \ideal^\star.
\]
This says that $F'\cap W'$ is nowhere dense in $\X(\ideal)$, which is a contradiction, as by construction, $F'\cap V$ is $\rho_\ideal$-dense in $V$ and $W'\cap U\subseteq V$ is a non empty $\rho_\ideal$-open set (it is non empty as $U$ is $\rho_\mathcal{\ideal^\star}$-dense). 

(3) follows immediately from (1) and (2). 
\end{proof}

The  natural bijection $\psi_n\mapsto \varphi_n$ is not continuous (neither is its inverse), however it has some form of semi-continuity as we show below. 

\begin{proposition}
Let $\ideal$ be an ideal over $\nat$. Let $\Gamma:\Y\to \X$ given  by $\Gamma (\psi_n)=\varphi_n$. Let $\alpha\in \ideal $ and $p\in \nat$.  Then $\Gamma^{-1} ((\alpha, p)^+\cap \X)$ is open in $\Y(\ideal)$. In general, if $V$ is a $\rho_\ideal$-basic open set, then there is  $D\subseteq \Y$  closed discrete in $\Y(\ideal)$ and an $\rho_\ideal$-open set $W$ such that  $\Gamma^{-1}(V\cap \X)= (W\cap \Y)\cup D$. 
\end{proposition}

\begin{proof}
Let $\alpha\in \ideal$ and $p\in\nat$.
Let $O=\{\psi_n:\; \varphi_n \in (\alpha, p)^+\}$. We need to show that $O$  is open in $\Y(\ideal)$.  Let $F=\alpha^{-1}(1)$. Since $((\alpha, p)^+\cap \Y)\setminus \widehat{F}\subseteq O\subseteq (\alpha, p)^+\cap \Y$, there is $A\subseteq F$ such that $O =((\alpha, p)^+\cap \Y)\setminus \widehat{A}$. As $A\in \ideal$, then by Lemma  \ref{Lemanodec}, $\widehat{A}$ is closed discrete in $\Y(\ideal)$. Thus $O$ is open in $\Y(\ideal)$. On the other hand, $\{\psi_n: \varphi_n\in (\alpha,p)^-\}= ((\alpha,p)^-\cap \Y)\cup ( \{\psi_n: \varphi_n\in (\alpha,p)^{-}\}\cap \widehat{F})$.
\end{proof}

The derivative operator on $\Y(\ideal)$ can be characterized as follows. 

\begin{proposition}
Let $\ideal$ be an ideal over $\nat$ and $A\subseteq \nat$. Then $\psi_l$ is a $\rho_\ideal$-accumulation point of $\widehat{A}$ if, and only if, for every non empty $\rho_{\ideal}$-open set $V$ with $\psi_l\in V$ we have
\[
\{n\in A:\;\varphi_n\in V\}\nin \ideal.
\]
\end{proposition}
\begin{proof} 
Let $V$ be a $\rho_{\ideal}$-open set with $\psi_l\in V$.  Suppose $F=\{n\in A:\;\varphi_n\in V\}\in \ideal$. Then by Lemma \ref{Lemanodec}, 
$\widehat{F}$ is closed discrete in $\Y(\ideal)$ which is a contradiction  as $\psi_l$ is an accumulation point of $\widehat{F}$. 
Conversely, let $V$ be a basic $\rho_\ideal$-open set containing $\psi_l$.  By Lemma  \ref{simetricdif} the following set
belongs to $\ideal$:
\[
E=\{n\in\nat: \varphi_n\in V \text{ and } \psi_n\nin V\}.
\]
We also have 
\[
F=\{n\in A:\;\varphi_n\in V\}\subseteq \{n\in A: \varphi_n\in V \text{ and } \psi_n\in V\}\cup E.
\]
Since $E\in \ideal$ and by hypothesis $F\nin \ideal$, then there are infinitely many  $n\in A$ such that $\psi_n\in V$ and we are done. 

\end{proof}

Now we show that the spaces $\X(\ideal)$ and $\Y(\ideal)$ are not homeomorphic in general. 

\begin{proposition}
Let $\ideal$ be a tall ideal over $\nat$. There are no non trivial convergent sequences in $\Y(\ideal)$. In particular, $\Y(\ideal)$ is not homeomorphic to $\X(\ideal)$.
\end{proposition}

\begin{proof}
Let $A\subseteq \nat$ be an infinite set. We will show that  $\widehat{A}=\{\psi_n:\; n\in A\}$ is not convergent in $\Y(\ideal)$.  Since $\ideal$ is tall, pick $B\subseteq A$ infinite  with $B\in \ideal$.   Then $\widehat{B}$ is closed discrete in $\Y(\ideal)$ (by Lemma \ref{Lemanodec}).  Thus $\widehat{A}$ is not convergent.
From this, the last claim follows since $\X(\ideal)$ has plenty of convergent sequences (see Proposition \ref{converseque}).
\end{proof}

Next result shows that our spaces are analytic. 

\begin{lemma}
\label{staranalytic}
Let $\ideal$ be an analytic ideal over $\nat$. Then $\ideal^{\star}$ is analytic.
\end{lemma}
	
\begin{proof}
The argument is analogous to that of the Lemma 4.8 of \cite{Todoruzca2014}. We include a sketch of it for the sake of completeness.  First, we recall a result from \cite{Todoruzca2014} (see Lemma 4.7).

\bigskip

\noindent{\em Claim:} Let $J$  be an infinite set. Then  $M \subseteq 2^{J}$ is nowhere dense if, and only if, there is $C \subseteq J$ countable  such that $M\restriction C=\{ x \restriction C: x \in M \}$ is nowhere dense in  $2^C$ 

\bigskip

Let $Z$  be the set of all  $z \in (\cantor \times \nat)^{\nat}$ such that  $z(k) \neq z(j)$ for all  $k \neq j$ and  $\{z(k): k \in \nat \} \subseteq \ideal \times \nat$. Since $\ideal$ is an analytic set, then  $Z$ is an  analytic subset of  $(\cantor \times \nat)^{\nat}$.
		
Consider the following relation $R \subseteq \mathcal{P}(\nat) \times (\cantor \times \nat)^{\nat}$:
$$
(F,z) \in R \Leftrightarrow \; z\in Z \;\mbox{and }\; \{\varphi_n \restriction \{z(k): k \in \nat \} : n \in F \} \text{ is nowhere dense in  }2^{ \{z(k):\; k \in \nat \} }.
$$
Then $R$ is an analytic set.
From the claim above,  we have
$$ 
F \in \ideal^{\star} \Leftrightarrow (\exists z \in (\cantor \times \nat)^{\nat}) R(F,z).
$$
Thus, $\ideal^{\star}$ is analytic.
\end{proof}

Finally, we can show one of our main results. 
Let us define a sequence $(\ideal^k)_{k \in \nat}$ of ideals on $\nat$ as follows:

$$\ideal^k=\left\{
\begin{array}{cl}
\cantor, & \text{if }k=0,  \\
(\ideal^{k-1})^{\star}, & \text{if }k>0. \\
\end{array}
\right.$$

Notice that $\ideal^{k+1} \subsetneq \ideal^{k}$ for each $k \in \nat$ by  Lemma \ref{IssinIs}.  

\begin{theorem}\label{Ykesnod}
For all $k>0$, $\Y(\ideal^k)$ is analytic, nodec and crowded.
\end{theorem}

\begin{proof}
That $\Y(\ideal^k)$ is  analytic and nodec follows from  Lemmas  \ref{staranalytic}, \ref{complexity}, \ref{Lemanodec} and \ref{IssinIs}.
Since $\ideal^k\subseteq \ideal_{nd}$, then by Lemma \ref{crowded}, $\Y(\ideal^k)$ is crowded.
\end{proof}

\bigskip

Thus we do not know whether $\Y(\ideal^\star)$ is nodec for ideals such that $\ideal^\star \not\subseteq \ideal$. The reason is that it is not clear if part (2) in Lemma \ref{Lemanodec} holds in general without the assumption that $\ideal^\star\subseteq \ideal$. In this respect, we only were able to show the following. 

\begin{lemma}
Let $\ideal$ be an ideal on $\nat$ such that  $\ideal\subseteq \ideal^\star$.  Let $A\subseteq \nat$. Then 

\begin{enumerate}
\item Let $V$ be a non empty $\rho_\ideal$-open set. If $A'$ is $\rho_\ideal$-dense in  $V$, then  $\widehat{A}$ is  $\rho_\ideal$-dense in  $V$.
    
\item  If $\widehat{A}$ is nowhere dense in $\Y(\ideal)$, then $A'$ is nowhere dense in $\X(\ideal)$ (i.e., $A
\in \ideal^\star$).
In particular,  if $\widehat{A}$ is nowhere dense in $\Y(\ideal)$, then $\widehat{A}$ is closed discrete in $\Y(\ideal^\star)$.
\end{enumerate}
\end{lemma}

\begin{proof}
(1) Let $V$ be a  non empty $\rho_\ideal$-open set and suppose  $A'$ is $\rho_\ideal$-dense in $V$. Let $W$ be a  basic $\rho_\ideal$-open set with $W\subseteq V$. We need to find $n\in A$ such that  $\psi_n\in W$. 
By Lemma  \ref{simetricdif} the following set
belongs to $\ideal$:
\[
E=\{n\in\nat: \varphi_n\in W \text{ and } \psi_n\nin W\}.
\]
As $\ideal\subseteq \ideal^\star$, then $E'$ is nowhere dense in $\X(\ideal)$. Since $A'$ is dense in $V$, then $A'\cap W\not\subseteq E'$. 
Let $n\in A\setminus E$ such that $\varphi_n\in W$. As $n\nin E$, then $\psi_n\in W$. 

(2) Follows from (1)   and part (1) in Lemma \ref{Lemanodec}. 
\end{proof}

Now we compare the dense sets in $\Y(\ideal)$ and $\X(\ideal)$. 

\begin{lemma}
\label{denytodenxstar}
Let $\ideal$ be an ideal on $\nat$ such that $\Y$ is dense in $(2^{\cantor\times\nat},\rho_\ideal)$ and $D\subseteq \nat$. If $\widehat{D}$ is dense in $\Y(\ideal)$, then  $D'$ is  dense in $\X(\ideal)$. 
\end{lemma}

\begin{proof}
Suppose $\widehat{D}$ is dense in $\Y(\ideal)$. Let $V$ be a  basic $\rho_\ideal$-open set. We need to find $n\in D$ such that  $\varphi_n\in V$. 
By Lemma  \ref{simetricdif} the following set
belongs to $\ideal$:
\[
E=\{n\in\nat: \varphi_n\nin V \text{ and } \psi_n\in V\}.
\]
Let $F=\{n\in D:\; \psi_n\in V\}$. Since $\widehat{D}$ is $\rho_{\ideal}$-dense, then $F\nin \ideal$ (by part (1) of Lemma \ref{Lemanodec} and the assumption that $Y$ is dense in $(2^{\cantor\times\nat},\rho_\ideal)$). Thus there is $n\in F\setminus E$. Then $\psi_n\in V$  and $\varphi_n\in V$. 

\end{proof}

Observe that $\fin\subseteq \fin^\star\subseteq \fin^{\star\star}\subseteq \cdots\subseteq \ideal^k$ for all $k$.   Notice that $\fin^\star$ is isomorphic to $\nwd (\mathbb{Q})$  as $\X(\fin)$ is homeomorphic to $\mathbb{Q}$. The following is a natural  and intriguing question.

\begin{question}
Is $\Y(\fin^\star)$ nodec?
\end{question}

It is unclear when an ideal $\ideal$ satisfies either $\ideal\subseteq \ideal^\star$ or  $\ideal^\star\subseteq \ideal$. The following question asks a concrete instance of this problem. 

\begin{question}
Two ideals  that naturally extend $\fin$ are $\{\emptyset\}\times\fin$ and $\fin\times \{\emptyset\}$ (where $\times$ denotes the Fubini product).  Let $\ideal$ be any of those  two ideals. Is $\ideal\subseteq \ideal^\star$? 
\end{question}

\subsection{SS property in $\Y(\ideal)$} 

We do not know whether $\Y(\ideal_{nd})$ is $SS$. However, we show below that  $\Y(\ideal^k)$ is not $wSS$ for all $k>1$, this was the reason to introduce the ideals $\ideal^\star$.

We need an  auxiliary result.

\begin{lemma} \label{densostar}
Let $\ideal$ be an ideal over $\nat$ such that $\ideal^{ \star} \subseteq \ideal$. Let $V$ be a non empty $\rho_{\ideal^\star}$-open set and  $D \subseteq \mathbb{N}$. If $D'$ is $\rho_\ideal$-dense in $V$, then $\widehat{D}$ in $\rho_{\ideal^\star}$-dense in $V$. 
\end{lemma}

\begin{proof}
Let $V$ be a non empty $\rho_{\ideal}$-open set and suppose that $D'$ is $\rho_\ideal$-dense in $V$. Let $W$ be  a $\rho_{\ideal^\star}$-basic open set such that $W\subseteq V$. We need to show that there is $n\in D$ such that $\psi_n\in W$. By Lemma  \ref{simetricdif} the following set belongs to $\ideal^\star$:
\[
E=\{n\in\nat: \varphi_n\in W \text{ and } \psi_n\nin W\}.
\]
Since $W$ is also $\rho_\ideal$-open (as $\ideal^\star \subseteq \ideal$) and $E'$ is $\rho_\ideal$-nowhere dense, then there is a non empty $\rho_\ideal$-open set $V_1\subseteq W$ such that $V_1\cap E'=\emptyset$. Since $D'$ is $\rho_\ideal$-dense in $V$, there is $n\in D$ such that $\varphi_n\in V_1$. Notice that $n\nin E$. Since $\varphi_n\in W$,  then $\psi_n\in W$. 
\end{proof}

\begin{theorem}\label{Yknoss}
Let $\ideal$ be an ideal over $\nat$ such that $\ideal^\star\subseteq \ideal$. Then $\Y(\ideal^{\star\star})$ is not  $wSS$.
\end{theorem}

\begin{proof} Notice that $\X$ is $\rho_{\ideal}$ crowded (see Lemma \ref{crowded}). Also, observe that $\ideal^{\star\star}\subseteq \ideal^\star$  (see  Lemma \ref{IssinIs}). 
Let $(U_n)_{n \in \nat}$ be a pairwise disjoint sequence of non empty $\rho_{\ideal^{\star\star}}$-open sets. Let $A_n = \{ m \in \nat: \varphi_m \in U_n \}$. It is clear that $A_n \notin \ideal^{\star}$ for each $n \in \nat$. It is easy to verify that the sequence $(A_m)_m$ witnesses  that  $\ideal^\star$ is not \ppoint.  Let $D_n=\mathbf{D}(A_n)$, as defined in Lemma \ref{DA}.  Let 
$$
E_n= \{\psi_m \in Y:\; \varphi_m \in D_n  \}.
$$
We claim that the sequence $(E_n)_{n \in \mathbb{N}}$ witnesses that the space $\Y(\mathcal{I}^{\star\star})$ is not $wSS$.  In fact, since $A_n \notin \ideal^{\star}$, then $D_n$ is dense in $\X(\ideal^{\star})$ (by Lemma \ref{DA}), so $E_n$ is dense in $\Y(\ideal^{\star\star})$ (by  Lemma \ref{densostar}). Let $K_n \subseteq E_n$  be a finite set and $L_n= \{ \varphi_m:  \psi_m \in K_n \}$ for each $m\in \nat$. Since $A_n\nin\ideal^{\star}$ and $\ideal^{\star}$ is not \ppoint, then,  by the proof of Theorem \ref{xinoSS}, $L=\bigcup_{n \in \mathbb{N}} L_n$ is  nowhere dense in $\X(\mathcal{I}^{\star})$. Thus $L\in \ideal^{\star\star}$. Therefore $\widehat{L}=\bigcup_{n \in \mathbb{N}} K_n$ is closed discrete in $\Y(\mathcal{I}^{\star\star})$ (by  Lemma \ref{Lemanodec}).
		
\end{proof}

We have seen in Theorem \ref{Ykesnod} that $\Y(\ideal^k)$ is nodec for every $k\geq 1$. From  Theorem \ref{Yknoss}  we have the following.

\begin{corollary}
$\Y(\ideal^{k})$ is not $wSS$ for every $k>1$.
\end{corollary}

Recall that $\ideal^{1}$ is $\ideal_{nd}$. We do not know  whether  $\Y(\ideal_{nd})$ is  SS. We only know the following. Suppose   $\widehat{D_n}=\{\psi_m: m \in D_n \}$ is open dense  in  $\Y(\mathcal{I}_{nd})$, for every $n\in \mathbb{N}$. Then  there is $F_n \subseteq D_n$ finite for each $n$ such that  $\bigcup_{n \in \mathbb{N}} \widehat{F_n}$ is  dense. 

\begin{question}
Is there an  ideal $\ideal$ on $\nat$ such that  $\ideal\subseteq \ideal^\star$ and  
$\Y(\ideal^\star)$  is $wSS$? In particular, is $\Y(\fin^\star)$ $wSS$?
\end{question}

\subsection{\qpoint\ in $\Y(\ideal)$}

We shall prove that for certain kind of ideals, $\Y(\ideal)$ is not $q^+$. 
We use a construction quite similar to that in the proof of Theorem \ref{eqclq}. 

We recall that in the proof of Lemma \ref{xnoq} we have introduced the following property: Let $m \in \mathbb{N}$. We say that $\varphi \in \X$ has the property $(*^m)$ if there are $k \in \mathbb{N}$, $s_i \in 2^{m+1}$   $(i=1,...,k)$  finite sequences and $m_i \leq m$ $(i=1,...,k)$ natural numbers such that $ \varphi= \bigcup _{i=1}^{k} [s_i]\times \{m_i\}$ and if $m_i=m_j$ with $i \neq j$, then $s_i \restriction m \neq s_j \restriction m$. 

\begin{lemma}
\label{xnoq2}
Let $\ideal$ be an ideal over $\nat$ such that  $\ideal\subseteq \ideal_{nd}$.   
Let
$$
A_m = \{ \varphi \in \X: \varphi \text{ has the property }(*^m)  \}
$$
and 
\[
B_m=\{\psi_n\in \Y:\; \varphi_n \in A_m\}.
\]

Let $L= \{n \in \mathbb{N}: \varphi_n \notin \bigcup _{m \in \mathbb{N}} A_m \}$ and suppose there is an infinite set $L'=\{m_k:\;k\in \nat\}\subseteq L$ such that $L'\in \mathcal{I}$.
Let 
\[
B=\bigcup_k B_{m_k}.
\]
Let $q\in \nat$ be such that $\varphi_q=\cantor\times \{0\}$. 
Then 
\begin{enumerate}
\item $B$ is dense in $\Y(\ideal)$ and, in particular, $\psi_q \in cl_{\rho_\ideal} (B)$. 

\item   Let $S\subseteq B$ be such that $S\cap B_{m_k}$ has at most one element for each $k$,  then $\psi_q\not\in cl_{\rho_\ideal} (S)$. 
\end{enumerate}

\end{lemma}

\proof
(1) Let $A=\bigcup_k A_{m_k}$. By Lemma \ref{xnoq},   $A$ is dense in $\X$. Thus by Lemma \ref{densostar},  $B$ is dense in $\Y(\ideal_{nd})$ (recall that $\ideal_{nd}=(\cantor)^\star$). As
 $\ideal\subseteq \ideal_{nd}$, then $B$ is also dense in $\Y(\ideal)$.
 
 \medskip
 
(2)   Let  $S=\{\psi_{n_k}:\; k\in \nat\}$  be such that $\psi_{n_k}\in B_{m_k}$ for all $k\in \nat$.  We will show that $\psi_q\nin  cl_{\rho_\ideal} (S)$.
 
 Let $\alpha\in\cantor$ be defined as follows: If $0 \in L'$ and $[\langle 0\rangle] \times \{0\} \subseteq \varphi_{n_0} $, then  $\alpha(0)=1$. Otherwise,  $\alpha(0)=0$. 	For  $n>1$,
$$
\alpha (n)=\left\{
				\begin{array}{cl}
				1, & \text{if }n \in L' \text{ , } \text{$n=m_k$ for some $k$} \\ &  \text{ and }[\langle \alpha(0),...,\alpha(n-1),0\rangle] \times \{0\} \subseteq \varphi_{n_k}. \\
				0, & \text{ otherwise.} 
				\end{array}
				\right.
$$
Observe that  $\alpha \in \mathcal{I}$,  as  $\alpha^{-1}(1) \subseteq L' \in \mathcal{I}$.
				
It is clear that $\psi_q\in (\alpha,0)^+$. To finish the proof, it suffices to show that  $(\alpha,0) \notin \bigcup_{k \in \mathbb{N}} \psi_{n_k}$. Suppose, towards a contradiction, that there is $l \in \mathbb{N}$ such that  $(\alpha,0) \in \psi_{n_l}$, that is,  $(\alpha,0) \in \varphi_{n_l} \cup ([n_l] \times \mathbb{N})$.  There are two cases to be considered. 

(i) Suppose $\alpha(n_l)=1$. Then $n_l\in L'$ and  thus $\varphi_{n_l}\nin A_{m_l}$ which contradicts that $\psi_{n_l}\in B_{m_l}$. 

\medskip

(ii)  Suppose $\alpha(n_l)=0$ and thus $(\alpha, 0)\in \varphi_{n_l}$.  Let $\varphi_{n_l}= \bigcup_{i=1}^{r} [s_i] \times \{p_i\}$ with  $s_i \in 2^{m_l+1}$.  Then $\alpha \in [s]$, where  $s$ is $s_i$ for some $i$ with $p_i=0$. Hence $\alpha(n)=s(n)$ for all  $n \leq m_l$.  We consider two cases. Suppose $\alpha(m_l)=1$. Then $s(m_l)=1$. Let $t$ be such that $s=t\widehat{\;} 1$. Then by the definition of $\alpha$, we have that  $[t\widehat{\;}0]\times\{0\}\subseteq \varphi_{n_l}$. But also 
$[s]\times\{0\}=[t\widehat{\;}1]\times\{0\}\subseteq \varphi_{n_l}$ which contradicts that $\varphi_{n_l}\in A_{m_l}$ (i.e. that it has property $(*^{m_l}))$.  Now suppose that $\alpha(m_l)=0$.  Then $[s]\times\{0\}=[t\widehat{\;}0]\times\{0\}\not\subseteq \varphi_{n_l}$, but this contradicts that  $[s]\times\{0\}$ is $[s_i]\times\{p_i\}$ for some $i$.
\endproof

From the previous lemma we immediately get the following.

\begin{theorem}\label{YI no q}
Let $\ideal$ be a tall ideal over $\nat$ such that  $\mathcal{I} \subseteq \mathcal{I}_{nd}$. Then $\Y(\mathcal{I})$ is not $q^+$.
\end{theorem}

\begin{question}
Is there an ideal (necessarily non tall) different from $\fin$ such that $\Y(\ideal)$ is \qpoint? Two natural candidates are $\{\emptyset\}\times\fin$ and $\fin\times \{\emptyset\}$. 
\end{question}

Finally, we have the following. 

\begin{theorem}
$\Y(\ideal^k)$ is a  non SS, non  \qpoint\ nodec regular space with analytic topology for every $k>1$.
\end{theorem}

\bigskip

\noindent {\bf Acknowledgment:}
We are  thankful to the referee for his (her)  comments that improved the presentation of the paper. 

\bibliographystyle{plain}

\begin{thebibliography}{10}

\bibitem{BarmanDow2011}
D.~Barman and A.~Dow.
\newblock Selective separability and {$SS^+$}.
\newblock {\em Topology Proc.}, 37:181--204, 2011.

\bibitem{BarmanDow2012}
D.~Barman and A.~Dow.
\newblock Proper forcing axiom and selective separability.
\newblock {\em Top. and its Appl.}, 159(3):806 -- 813, 2012.

\bibitem{Bella2009}
A.~Bella, M.~Bonanzinga, and M.~Matveev.
\newblock Variations of selective separability.
\newblock {\em  Top. and its Appl.}, 156(7):1241 -- 1252, 2009.

\bibitem{Bella_et_al2008}
A.~Bella, M.~Bonanzinga, M.~Matveev, and V.~Tkachuk.
\newblock Selective separability: general facts and behavior in countable
  spaces.
\newblock {\em Topology Proceedings}, 32:15--30, 2008.

\bibitem{Bella2013}
A. Bella.
\newblock When is a {P}ixley-{R}oy hyperspace {$SS^+$}?
\newblock {\em  Top. and its Appl.}, 160(1):99 -- 104, 2013.

\bibitem{CamargoUzca2018b}
J.~Camargo and C.~Uzc\'ategui.
\newblock Selective separability on spaces with an analytic topology.
\newblock {\em Topology and its applications}, 248(1):176--191, 2018.

\bibitem{DTTW2002}
A.~Dow, M.~G. Tkachenko, V.~V. Tkachuk, and R.~G. Wilson.
\newblock Topologies generated by discrete subspaces.
\newblock {\em Glas. Math. Ser. III}, 37(57):187--210, 2002.

\bibitem{Gruenhage2011}
G. Gruenhage and M. Sakai.
\newblock Selective separability and its variations.
\newblock {\em Top. and its Appl.}, 158(12):1352 -- 1359, 2011.


\bibitem{HMTU2017}
M.~Hru\u{s}\'ak, D.~Meza-Alc\'antara, E.~Th\"{u}mmel, and C.~Uzc\'ategui.
\newblock {R}amsey type properties of ideals.
\newblock {\em Annals of Pure and Applied Logic}, 168(11):2022--2049, 2017.

\bibitem{Kechris94}
A.~S. Kechris.
\newblock {\em Classical Descriptive Set Theory}.
\newblock Springer-Verlag, 1994.

\bibitem{Njastad1965}
O.~Nj\aa stad.
\newblock On some classes of nearly open sets.
\newblock {\em Pacific. J. Math.}, 15:961--970, 1965.

\bibitem{Reposvetal2010}
D. Repov\v{s} and L. Zdomskyy.
\newblock On {$M$}-separability of countable spaces and function spaces.
\newblock {\em Topology Appl.}, 157(16):2538--2541, 2010.

\bibitem{Scheeper99}
M.~Scheepers.
\newblock Combinatorics of open covers {VI}: Selectors for sequences of dense
  sets.
\newblock {\em Quaestiones Mathematicae}, 22(1):109--130, 1999.

\bibitem{todoruzca}
S.~Todor\v{c}evi\'c and C.~Uzc\'ategui.
\newblock Analytic topologies over countable sets.
\newblock {\em Top. and its Appl.}, 111(3):299--326, 2001.

\bibitem{Todoruzca2000}
S.~Todor\v{c}evi\'c and C.~Uzc\'ategui.
\newblock Analytic $k$-spaces.
\newblock {\em Top. and its Appl.}, 146-147:511--526, 2005.

\bibitem{Todoruzca2014}
S.~Todor\v{c}evi\'c and C.~Uzc\'ategui.
\newblock A nodec regular analytic space.
\newblock {\em Top. and its Appl.}, 166:85--91, 2014.

\bibitem{Vand}E. Van Douwen. \textit{Applications of maximal topologies.} Top. and its Appl., 51:125-139, 1993.

\end{thebibliography}

\end{document}